\documentclass[a4paper,10pt,twoside]{article}
\usepackage{pst-all}
\usepackage{graphicx}
\usepackage{fancyhdr}
\usepackage{amsmath}
\usepackage{epsfig} 
\usepackage{psfrag} 
\usepackage{amsmath,amsfonts,amssymb,amsthm} 
\usepackage{bbm} 

\newtheorem{Theorem}{Theorem}[section] 
\newtheorem{Corollary}[Theorem]{Corollary} 
\newtheorem{Lemma}[Theorem]{Lemma}

\theoremstyle{definition} 

\newcommand{\CHS}{\mbox{CHS}}

\newcommand{\FS}{\mbox{FS}}
\newcommand{\MAT}{P_{\mathcal{I}}^c(E)}

\newcommand{\rank}{r_{\mathcal{I}}}

\title{On the cardinality constrained matroid polytope}
\author{Jean Fran\c{c}ois Maurras\thanks{Laboratoire d'Informatique Fondamentale, UMR 6166, Universit\'e de la Mediterran\'ee, Facult\'e des sciences de Luminy, 163 Avenue de Luminy, 13288 Marseille, France, e-mail : jean-francois.maurras@lif.univ-mrs.fr}  and R\"udiger Stephan\thanks{Institut f\"ur Mathematik, Technische Universit\"at Berlin, Stra{\ss}e des 17. Juni 136, 10623 Berlin, e-mail : stephan@math.tu-berlin.de}}
\date{}
\begin{document}

\maketitle

\begin{abstract}
Given a combinatorial optimization problem $\Pi$ and an increasing finite sequence $c$ of natural numbers, we obtain a cardinality constrained version $\Pi_c$ of $\Pi$ by permitting only those feasible solutions of $\Pi$ whose cardinalities are members of $c$.
We are interested in polyhedra associated with those problems, in particular in inequalities that cut off solutions of forbidden cardinality. Maurras~\cite{Maurras77} and Camion and Maurras~\cite{CM82} introduced a family of inequalities, that we call {\em forbidden set inequalities}, which can be used to cut off those solutions. However, these inequalities are in general not facet defining for the polyhedron associated with $\Pi_c$.
In \cite{KS} it was shown how one can combine integer
characterizations for cycle and path polytopes and a modified form of forbidden set inequalities to give facet defining integer
representations for the cardinality restricted versions of these
polytopes. Motivated by this work, we apply the same approach on the
matroid polytope. It is well known that the so-called rank
inequalities together with the nonnegativity constraints provide a complete
linear description of the matroid polytope (see
Edmonds~\cite{Edmonds1971}). By essentially adding the 
forbidden set inequalities in an appropriate form, we obtain a complete linear
description of the cardinality constrained matroid
polytope which is the convex hull of the incidence vectors of those
independent sets that have a feasible cardinality. Moreover, we
show how the separation problem for the forbidden set
inequalities can be reduced to that for the rank inequalities. We also give
necessary and
sufficient conditions for a forbidden set inequality to be facet
defining.
\end{abstract}

\section{Introduction}
Let $E$ be a finite set and $\mathcal{I}$ a subset of the power set of
$E$. The pair $(E,\mathcal{I})$ is called an \emph{independence
  system} if (i) $\emptyset \in \mathcal{I}$ and (ii) whenever $I \in
\mathcal{I}$ then $J \in \mathcal{I}$ for all $J \subset I$. If $I
\subseteq E$ is in $\mathcal{I}$, then $I$ is called an
\emph{independent set}, otherwise it is called a \emph{dependent set}.
Dependent sets $\{e\}$ with $e\in E$ are called \emph{loops}.
For any set $F \subseteq E$, $B \subseteq F$ is called a \emph{basis}
of $F$ if $B \in \mathcal{I}$ and $B \cup \{e\}$ is dependent for all
$e \in F \setminus B$. The \emph{rank} of $F$ is defined by
$\rank(F):= 
\max \{|B| : B \mbox{ basis of } F\}$. The set of all bases $B$ of $E$
is called a \emph{basis system}. There are
many different ways to characterize when an independence system is a
matroid. For our purposes the following definition will be most
comfortable. $(E,\mathcal{I})$ is called a \emph{matroid}, and then it
will be denoted by $M=(E,\mathcal{I})$, if 
\begin{equation*} \label{def-matroid}
\mbox{(iii)} \hspace{1cm}I,J \in \mathcal{I}, |I|<|J| \; \Rightarrow \;
\exists \,K \subseteq J 
\setminus I : |I \cup K|=|J|, \; K \cup I \in \mathcal{I}.
\end{equation*}
Equivalent to (iii) is the requirement that for each $F \subseteq E$
all its bases have the same cardinality. Throughout the paper we deal
only with loopless matroids. The results of the paper can be easily
brought forward to matroids containing loops.

Let $M=(E,\mathcal{I})$ be a matroid. A set $F \subseteq E$ is said to
be \emph{closed} if $\rank(F)<\rank(F \cup \{e\})$ for all $e \in E
\setminus F$ and \emph{inseparable} if there are no sets $F_1 \neq \emptyset
\neq F_2$ with $F_1 \,\dot{\cup} \,F_2=F$ such that $\rank(F_1)+
\rank(F_2) \leq \rank(F)$.

Given any independence system $(E,\mathcal{I})$ and any weights $w_e
\in \mathbb{R}$ on the elements $e \in E$, the combinatorial
optimization problem $\max w(I) , I \in \mathcal{I}$, where
$w(I):=\sum_{e \in I} w_e$, is called the \emph{maximum weight
  independent set problem}. The convex hull of the incidence vectors
of the feasible solutions $I \in \mathcal{I}$ is called 
the \emph{independent set polytope} and will be denoted by
$P_{\mathcal{I}}(E)$. If $(E,\mathcal{I})$ is a matroid, then
$P_{\mathcal{I}}(E)$ is  also called the \emph{matroid polytope}.

As it is well known, the maximum weight independent set problem on a
matroid can be solved to optimality with the greedy
algorithm. Moreover, the matroid polytope $P_{\mathcal{I}}(E)$ is
determined by the rank inequalities and the nonnegativity constraints (see
Edmonds~\cite{Edmonds1971}), i.e.,  $P_{\mathcal{I}}(E)$ is the 
set of all points $x \in \mathbb{R}^E$ satisfying 
\begin{equation} \label{matroid}
\begin{array}{rcll}
\sum\limits_{e \in F} x_e & \leq & \rank(F) & \mbox{for all }
\emptyset \neq F \subseteq E,\\ 
x_e & \geq & 0 & \mbox{for all } e \in E.
\end{array}
\end{equation}
The rank inequality associated with $F$ is facet defining for
$P_{\mathcal{I}}(E)$ if and only if $F$ is closed and inseparable (see
Edmonds~\cite{Edmonds1971}).

Let $c=(c_1,\dots,c_m)$ be a finite sequence of integers with
$0 \leq c_1 < c_2 < \ldots < c_m$. Then, the \emph{cardinality constrained
independent set polytope} $\MAT$ is defined to be the convex hull of the
incidence vectors of the independent sets $I \in \mathcal{I}$ with
$|I|=c_p$ for some $p \in \{1,\dots,m\}$, that is, $\MAT = \mbox{conv}
\{\chi^I \in \mathbb{R}^E : I \in \mathcal{I}, \, |I|=c_p \mbox{ for
  some } p \in \{1,\dots,m\}\}$. If $(E,\mathcal{I})$ is a matroid,
then $\MAT$ is called the \emph{cardinality constrained matroid
  polytope}. In the next section we will see that, if $(E,\mathcal{I})$ is a
matroid, then the associated combinatorial optimization problem $\max
w^T x,\; x \in \MAT$ can be solved in polynomial time. Since $c=(c_1,\dots,c_m)$ is
linked to a cardinality constrained optimization problem, it is called a
\emph{cardinality sequence.} Throughout the paper we assume that $m \geq 2$.

The underlying basic problem of cardinality restrictions can be completely described in terms of linear inequalities.
Given a finite set $B$ and a cardinality sequence $c=(c_1,\dots,c_m)$,
the set $\CHS^{c}(B):=\{F \subseteq B
: |F|=c_p \mbox{ for some } p\}$ is called a \emph{cardinality
  homogenous set system}.
The polytope associated with $\CHS^c(B)$, namely the convex hull of
the incidence vectors of elements of $\CHS^c(B)$, is completely described by
the \emph{trivial inequalities} $0 \leq z_e \leq 1$, $e \in B$, the
\emph{cardinality bounds} $c_1 \leq \sum_{e \in B} z_e \leq c_m$, and the
\emph{forbidden set inequalities}
\begin{equation} \label{FS}
\begin{array}{l}
(c_{p+1} - |F|) \sum\limits_{e \in F} z_e\;-\;(|F| - c_p) \sum\limits_{e \in B
  \setminus F} z_e \leq 
c_p(c_{p+1}-|F|) \\
\hspace{0.5cm} \mbox{for all } F \subseteq B \mbox{ with } c_p < |F| <
c_{p+1} \mbox{ for some } p \in \{1,\dots,m-1\}.
\end{array}
\end{equation}
This result is due to Maurras~\cite{Maurras77} and Camion and Maurras~\cite{CM82}.
Gr\"otschel~\cite{Groetschel} rediscovered inequalities~\eqref{FS} independently and proved the same result.

In ~\cite{KS} the authors investigated cardinality constrained
cycle and path problems. They observed that inequalities \eqref{FS} define very low dimensional faces of the associated polyhedra. However, with a modified version of the cardinality
forcing inequalities they were able to provide
characterizations of the integer points of cardinality constrained
cycle and path polytopes by facet defining inequalities. 

In our context ``modified version'' means to replace $|F|$ by $\rank(F)$.
To this end, consider, for instance, the cardinality constrained graphic matroid.
The independence system is the collection of all forests.
Figure 1 illustrates the support graph of an ordinary forbidden set inequality.
The set of bold edges, denoted by $F$, is of forbidden cardinality, since $9$ is not in the cardinality sequence $c=(3,5,12,14)$.
The forbidden set inequality associated with $F$ has coefficients $3$ on the bold edges and $-4$ on the dashed edges.
The right hand side is $15$. As it is not hard to see, none of the incidence vectors of forests of feasible cardinality satisfies the inequality at equality. However, if we fill up $F$ with further edges such that we obtain an edge set, say $F'$, of rank $9$, then the resulting inequality, which is illustrated in Figure 2, remains valid. Moreover, there are forests of cardinality $5$ and $12$ whose incidence vectors satisfy the resulting inequality at equality.

\begin{figure}
\includegraphics[width=0.45\textwidth]{graphicMatroid1.epsi} \hfill  \includegraphics[width=0.45\textwidth]{graphicMatroid2.epsi}
\end{figure}

With respect to $M=(E,\mathcal{I})$, $\MAT = \mbox{conv}
\{\chi^I \in \mathbb{R}^E : I \in \mathcal{I} \cap \CHS^c(E)\}$. By
default, we assume that $c_m \leq \rank(E)$.
Our main
result is that the system
\begin{gather}
 \nonumber \FS_F(x) := (c_{p+1}-\rank(F)) x(F)-
(\rank(F)-c_p)x(E \setminus F)  \leq
c_p(c_{p+1}-\rank(F))\\ 
\hspace{0.5cm}\mbox{for all $F \subseteq E$ with $c_p <
    \rank(F) < c_{p+1}$ for some $p \in \{0,\dots,m-1\}$,} \label{eq_FS}
\end{gather}

\vspace{-0.8cm}
\begin{align}
x(E) & \geq c_1,& \label{eq_lowerBound}\\
x(E) & \leq c_m, &\label{eq_upperBound}\\
x(F) & \leq \rank(F) &\hspace{2cm}\mbox{for all } \emptyset \neq F
\subseteq E,\label{eq_rankInequalities}\\ 
x_e & \geq 0 &\hspace{2cm}\mbox{for all } e \in E\label{eq_nn}
\end{align}
completely describes $\MAT$. Here, for any $I \subseteq E$
we set $x(I):=\sum_{e \in I} x_e$. Of course, each $x \in \MAT$ satisfies
$c_1 \leq x(E) \leq c_m$. Inequalities~\eqref{eq_FS} are called {\em rank induced forbidden set inequalities}.
The inequality $\FS_F(x)
\leq  c_p(c_{p+1}-\rank(F))$ associated with $F$, where
$c_p<\rank(F)<c_{p+1}$, 
is valid as can be seen as follows. The incidence vector of any $I \in
\mathcal{I}$ of cardinality at most $c_p$ satisfies the inequality,
since $|I \cap F| = \rank(I \cap F) \leq c_p$:
\begin{eqnarray*} (c_{p+1}-\rank(F)) \chi^I(F)-
(\rank(F)-c_p)\chi^I(E \setminus F) & \leq &
(c_{p+1}-\rank(F)) \chi^I(F)\\ 
& \leq & (c_{p+1}-\rank(F)) c_p.
\end{eqnarray*}
The incidence vector of any $I \in
\mathcal{I}$ of cardinality at least $c_{p+1}$ satisfies also the
inequality, since $\rank(I \cap F) \leq \rank(F)$
and thus $\rank(I \cap (E \setminus F)) \geq
c_{p+1}-\rank(F)$: 
\begin{eqnarray*} 
& & (c_{p+1}-\rank(F)) \chi^I(F)-
(\rank(F)-c_p)\chi^I(E \setminus F)\\
 & \leq &
(c_{p+1}-\rank(F)) \rank(F)
-(\rank(F)-c_p)\chi^I(E \setminus F)\\
 & \leq &
(c_{p+1}-\rank(F)) \rank(F)
-(\rank(F)-c_p) (c_{p+1}- \rank(F))\\
 & = &  c_p(c_{p+1}-\rank(F)).
\end{eqnarray*}
However, it is not hard to see that some incidence vectors of
independent sets $I$ with $c_p < |I| < c_{p+1}$
violate the inequality. 

When $M=(E,\mathcal{I})$ is the trivial matroid, i.e., all $F
\subseteq E$ are independent sets, then $\mathcal{I} \cap \CHS^c(E) =
\CHS^c(E)$. Thus, cardinality constrained matroids are a
generalization of cardinality homogenous set systems.

The paper is organized as follows. In Section 2 we prove that the
system \eqref{eq_FS}-\eqref{eq_nn} provides a complete linear
description of the cardinality constrained matroid polytope. Next, we
will give sufficient 
conditions for the rank induced forbidden set inequalities to be facet
defining. Finally, we show that the separation problem for the
rank induced forbidden set inequalities can be reduced to that for the rank
inequalities. This results in a polynomial time separation routine
based on Cunningham's separation algorithm for the rank inequalities.
In Section 3 we briefly discuss some consequences for cardinality
constrained combinatorial optimization problems and in particular for
the intersection of two cardinality constrained matroid polytopes.

\section{Polyhedral analysis of $\MAT$}
Let $M=(E,\mathcal{I})$ be a matroid. As already mentioned,
$P_{\mathcal{I}}(E)$ is determined by \eqref{matroid}.
For any natural number $k$, the independence system $M' :=(E,
\mathcal{I'})$ defined by $\mathcal{I'} := \{I \in \mathcal{I} : |I|
\leq k\}$ is again a matroid and is called the $k$-\emph{truncation}
of $M$. Therefore, the matroid polytope $P_{\mathcal{I'}}^c(E)$
associated with the $k$-truncation of $M$ is defined by system 
\eqref{matroid}, where the rank inequalities are indexed with
$\mathcal{I'}$ instead of $\mathcal{I}$. Following an argument of
Gamble and Pulleyblank~\cite{GP}, the only set of the
$k$-truncation which might be closed and inseparable with respect to the
truncation, but not with respect to the original matroid $M$ is $E$
itself, and the rank inequality associated with $E$ is the cardinality
bound $x(E) \leq k$. Hence, in context of the original matroid $M$,
$P_{\mathcal{I'}}^c(E)$ is described by
\begin{equation} \label{<=k}
\begin{array}{rcll}
x(F) & \leq & \rank(F) & \mbox{for all } \emptyset \neq F
\subseteq E,\\ 
x(E) & \leq & k,\\
x_e & \geq & 0 & \mbox{for all } e \in E.
\end{array}
\end{equation}
Of course, the connection to cardinality constraints is obvious, since
$P_{\mathcal{I'}}^c(E)= P_{\mathcal{I}}^{(0,\dots,k)}(E)$.
The basis system of $M'$ is the set of all bases $B$ of $E$ (with
respect to $M'$) and in case of $\rank(E) \geq r_{\mathcal{I'}}(E)$
the bases are all of cardinality $k$. Assuming $\rank(E) \geq
r_{\mathcal{I'}}(E)$, the associated polytope 
$$\mbox{conv}\{\chi^B \in
\mathbb{R}^E : B \mbox{ basis of $E$ with respect to $M'$}\}$$
is determined by
\begin{equation} \label{=k}
\begin{array}{rcll}
x(F) & \leq & \rank(F) & \mbox{for all } \emptyset \neq F
\subseteq E,\\ 
x(E) & = & k,\\
x_e & \geq & 0 & \mbox{for all } e \in E.
\end{array}
\end{equation}
On a basis system of a matroid one can optimize in polynomial
time by application of the greedy algorithm. Thus, for each member $c_p$
of a cardinality sequence 
$c=(c_1,\dots,c_m)$ an optimal solution $I^p$ of the linear
optimization problem 
$\max w(I),\; I \in \mathcal{I}, |I|=c_p$ can be found
in polynomial time. The best of the solutions $I^p, p=1,\dots,m$ with
respect to the linear objective $w$ is then the optimal solution of 
$\max w(I),\; I \in \mathcal{I} \cap \CHS^c(E)$. Since
$0 \leq c_1< \dots < c_m \leq \rank(E) \leq |E|$ and thus $m \leq
|E|$, it can be found by at most $|E|+1$ applications of the greedy
algorithm.   

These preliminary remarks are sufficient to present our main
theorem. In the sequel, we denote the rank function by $r$ instead of
$r_{I}$. 
Given a valid inequality $ax \leq a_0$
with $a \in \mathbb{R}^E$, $F \subseteq E$ is said to be \emph{tight} if $a
\chi^F = a_0$. A valid inequality $ax \leq a_0$ is \emph{dominated} by
another valid inequality $bx \leq b_0$, if $\{x \in \MAT : ax = a_0\}
\subseteq \{x \in \MAT : bx = b_0\}$. It is said to be \emph{strictly
  dominated} by $bx \leq b_0$, if $\{x \in \MAT : ax = a_0\}
\subsetneq \{x \in \MAT : bx = b_0\}$.

\renewcommand{\rank}{r}
\subsection{A complete linear description} 
\begin{Theorem}
The cardinality constrained matroid polytope $\MAT$ is completely
described by system \eqref{eq_FS}-\eqref{eq_nn}.
\end{Theorem}

\begin{proof}
Since all inequalities of system \eqref{eq_FS}-\eqref{eq_nn} are valid,
$P_{\mathcal{I}}^c(M)$ is contained in the polyhedron defined by
\eqref{eq_FS}-\eqref{eq_nn}. To show the converse, we
consider any valid inequality $bx \leq b_0$ for $P_{\mathcal{I}}^c(M)$
and associate with the inequality the following subsets of $E$:
\begin{eqnarray*}
P &:=& \{e \in E : b_e > 0\},\\
Z &:=& \{e \in E : b_e = 0\},\\
N &:=& \{e \in E : b_e < 0\}.\\
\end{eqnarray*}
We will show by case by case enumeration that the inequality $bx \leq
b_0$ is dominated by some inequality of the system
\eqref{eq_FS}-\eqref{eq_nn}. By definition, $E = P \dot{\cup}
Z \dot{\cup} N$, and hence, if $P=Z=N=\emptyset$, then $E=\emptyset$,
and it is nothing to show.
By a scaling argument we may assume that either $b_0=1$,
$b_0=0$, or $b_0=-1$.
\begin{enumerate}
\item[(1)] $b_0=-1$.
\begin{enumerate}
\item[(1.1)] $c_1=0$. Then $0 \in \MAT$, and hence $0=b \cdot 0 \leq -1$, a
  contradiction.
\item[(1.2)] $c_1>0$.    
\begin{enumerate}
\item[(1.2.1)] $P=Z=\emptyset, \,N \neq \emptyset$. Assume that there
  is some tight $I \in \mathcal{I}$ with $|I|=c_p$, $p \geq 2$. Then,
  for any $J \subset I$ with $|J|=c_1$ holds: $\chi^J \in \MAT$ and $b
  \chi^J > b \chi^I = -1$, a contradiction. Therefore, if any $I \in
  \mathcal{I} \cap \CHS^c(E)$ is tight, then $|I|=c_1$. Thus, $bx
  \leq -1$ is dominated by the cardinality bound $x(E) \geq c_1$.
\item[(1.2.2)]  $P \cup Z \neq \emptyset, \,N = \emptyset$. Then, $by \geq 0$
  for all $y \in \MAT$, a contradiction.
\item[(1.2.3)] $P \cup Z \neq \emptyset, \,N \neq \emptyset$. If
  $c_1 \leq r(P \cup Z)$, then there is some independent set $I
  \subseteq P \cup Z$ of cardinality $c_1$, and hence, $b \chi^I \geq
  0$, a contradiction. Thus, $c_1 > r(P \cup Z)$. Assume, for the sake
  of contradiction, that there is some tight independent set $J$ of
  cardinality $c_p$ with $p \geq 2$. If $J \subseteq N$, then the
  incidence vector of any $K \subset J$ with $|K|=c_1$ violates $bx
  \leq -1$. Hence, 
  $J \cap (P \cup Z) \neq \emptyset$. On the other hand, $J \cap N
  \neq \emptyset$ due to $c_p > c_1 > r(P \cup Z)$. However, by
  removing any $(c_p-c_1)$ elements in $N \cap J$, we obtain some
  independent set $K$ of cardinality $c_1$ whose incidence vector violates
  the inequality $bx \leq -1$, a contradiction. Therefore, if any $T
  \in \mathcal{I} \cap \CHS^c(E)$ is tight, then $|T|=c_1$. Thus, $bx
  \leq -1$ is dominated by the bound $x(E) \geq c_1$.
\end{enumerate}
\end{enumerate}
\item[(2)] $b_0=0$.
\begin{enumerate}
\item[(2.1)] $P \cup Z \neq \emptyset, \, N= \emptyset$. Then, either $bx
  \leq 0$ is not valid or $b=0$.
\item[(2.2)] $P=\emptyset, \, Z \cup N \neq \emptyset$. Then, $bx \leq 0$ is
  dominated by the nonnegativity constraints $x_e \geq 0$ for $e \in
  N$ or $b=0$.
\item[(2.3)] $P \neq \emptyset, \, N \neq \emptyset$.  
\begin{enumerate}
\item[(2.3.1)]  $c_1 >0$.  If $c_1 \leq r(P \cup Z)$,
  then there is 
  some independent set $I \subseteq P \cup Z$ with $I \cap P \neq
  \emptyset$ of cardinality $c_1$, and hence, $b \chi^I > 0$, a
  contradiction. Thus, $c_1 > r(P \cup Z)$. Assume, for the sake
  of contradiction, that there is some tight independent set $J$ of
  cardinality $c_p$ with $p \geq 2$. Since $c_p>c_1>r(P \cup Z)$ and
  $J$ is tight, $J \cap (P \cup Z) \neq \emptyset \neq J \cap N$.
  From here, the proof for this case can be finished as the proof for
  the case (1.2.3)  with $b_0=0$ instead of $b_0=-1$ in order to show
  that $bx \leq 0$ is dominated by the cardinality bound $x(E) \geq c_1$.
\item[(2.3.2)] $c_1 =0$. As in case (2.3.1), it follows immediately
  that $c_2 >r(P \cup Z)$, and if $I \in \mathcal{I} \cap \CHS^c(E)$ is
  tight, then $|I|=c_1=0$, that is, $I=\emptyset$, or
  $|I|=c_2$. Moreover, if $I  \in \mathcal{I}$ with $|I|=c_2$ is
  tight, then follows $|I \cap (P
  \cup Z)|= r(P \cup Z)$. Hence, $bx \leq b_0$ is dominated by the
  rank induced forbidden set inequality $\FS_F(x) \leq 0$ with $F=P \cup Z$.
\end{enumerate}
\end{enumerate}
\item[(3)] $b_0=1$.
\begin{enumerate}
\item[(3.1)] $P= \emptyset, \, Z \cup N \neq \emptyset$. Then, $b \leq
  0$, and hence $bx \leq 1$ is
  dominated by any nonnegativity constraint $x_e \geq 0$, $e \in
  E$. 
 \item[(3.2)] $P \cup Z \neq \emptyset, \, N =
   \emptyset$.
  Assume that there is some $I \in \mathcal{I}, I \notin \CHS^c(E)$
  with $|I| < c_m$ that violates $bx \leq 1$. Then, of course, all
  independent sets $J \supset I$ violate $bx \leq 1$, in particular,
  those $J$ with $|J|=c_m$, a contradiction. Hence, $bx \leq 1$ is not
  only a valid inequality for $\MAT$ but also for
  $P_{\mathcal{I}}^{(0,1,\ldots, c_m)}(E)$, that is, $bx \leq 1$ is
  dominated by 
  some inequality of the system \eqref{<=k} with $k=c_m$.
 \item[(3.3)] $P \neq \emptyset, \, N \neq \emptyset$. Let $p \in
   \{1,\dots,m\}$ be minimal such that there is a tight independent
   set $I^*$ of cardinality $c_p$. Of course, $c_p>0$, because otherwise
   $I^*$ could not be tight. If $p=m$, then $bx \leq 1$ is dominated by
   the cardinality bound $x(E) \leq c_m$, because then all tight $J \in
   \mathcal{I} \cap \CHS^c(E)$ have to be of cardinality
   $c_p=c_m$. So, let $0< c_p<c_m$. We distinguish 2 subcases.
   \begin{enumerate}
     \item[(3.3.1)] $c_p \geq r(P \cup Z)$. Suppose, for the
       sake of contradiction, that there is some tight independent set
       $I$ of cardinality $c_p$ such that $|I \cap
       (P \cup Z)| < r(P \cup Z)$. Then, $I \cap (P \cup Z)$ can
       be completed to a basis $B$ of $P \cup Z$, and since $|B|
       \leq |I|$, there is some $K \subseteq I \setminus B$ such
       that $I' := B \cup K \in \mathcal{I}$ and $|I'|=|I|$. $K$
       is maybe the empty set. Anyway, by construction, $I'$ is of
       cardinality $c_p$ and violates the inequality $bx \leq
       1$. Thus, $|I \cap (P \cup Z)| = r(P \cup Z)$. For the same
       reason, any tight $J \in \mathcal{I} \cap \CHS^c(E)$ satisfies
       $|J \cap (P \cup Z)| = r(P \cup Z)$, and since $p$ is minimal,
       $|J| \geq c_p$. Now, with similar arguments as in case (1.2.3)
       one can show that if $T  \in \mathcal{I} \cap \CHS^c(E)$ is
       tight, then $|T|=c_p$. Thus, $c_p=c_1 >0$ and $bx \leq 1$
       is dominated by the cardinality bound $x(E) \geq c_1$.
     \item[(3.3.2)] $c_p < r(P \cup Z)$. Following the argumentation
       line in (3.3.1), we see that $I \subseteq P \cup Z$ and $|I
       \cap P|$ has to be maximal for any tight independent set $I$ of
       cardinality $c_p$. Assume that $c_{p+1} \leq r(P \cup
       Z)$. Then, from any tight independent set $I$ with $|I|=c_p$ we
       can construct a tight independent set $J$ with $|J|=c_{p+1}$ by
       adding some elements $e \in Z$. However, it is not hard to see
       that there is no tight $K \in \mathcal{I} \cap \CHS^c(E)$ that
       contains some $e \in N$. Thus, when $c_{p+1} \leq r(P \cup Z)$,
       $bx \leq 1$ is dominated by the nonnegativity constraints $y_e
       \geq 0$, $e \in N$. Therefore, $c_{p+1} > r(P \cup Z)$. The
       following is now immediate: If $I \in \mathcal{I} \cap
       \CHS^c(E)$ is tight, then $|I|=c_p$ or $|I|=c_{p+1}$; if
       $|I|=c_p$, then $I \subset P \cup Z$, and if $|I|=c_{p+1}$,
       then $|I \cap (P \cup Z)|= r(P \cup Z)$ and $c_{p+1} > r(P \cup
       Z)$. Thus, $bx \leq 1$ is dominated by the rank induced forbidden set
       inequality $\FS_{P \cup Z}(x) \leq c_p(c_{p+1}-r(P \cup
       Z))$.
\end{enumerate}
\end{enumerate}
\end{enumerate}
\end{proof}

\subsection{Facets} \label{Sec:Facets}

We first study the facial structure of a single cardinality constrained matroid polytope $P_{\mathcal{I}}^{(k)}(E)$. All points of $P_{\mathcal{I}}^{(k)}(E)$ satisfy the equation $x(E)=k$, and hence, any inequality $x(F) \leq r(F)$ is equivalent to the inequality $x(E \setminus F) \geq k - r(F)$. Motivated by this observation, we introduce the following definitions. For any $F \subseteq E$, the number $r^k(F) := k - r(E \setminus F)$ is called the $k$-\emph{rank} of $F$. Due to the submodularity of $r$ we have $r^k(F_1) + r^k(F_2) \leq r^k (F)$ for all $F_1,F_2$ with $F = F_1 \dot{\cup} F_2$, and $F$ is said to be $k$-\emph{separable} if equality holds for some $F_1 \neq \emptyset \neq F_2$, otherwise $k$-\emph{inseparable}. Due to the equation $x(E)=k$, $\dim P_{\mathcal{I}}^{(k)}(E) \leq |E|-1$, and in fact, in the most cases we have equality. However, if $\dim P_{\mathcal{I}}^{(k)}(E) < |E|-1$, then at least one rank inequality $x(F) \leq r(F)$ with $\emptyset \neq F \subsetneq E$ is an implicit equation. As is easily seen, this implies that an inequality $x(F') \leq r(F')$ (or $x(F') \geq r^k(F')$) does not necessarily induce a facet of $P_{\mathcal{I}}^{(k)}(E)$, although $F$ is inseparable ($k$-inseparable). To avoid the challenges involved, we only characterize the polytopes $P_{\mathcal{I}}^{(k)}(E)$ of dimension $|E|-1$.

\begin{Lemma} \label{L:facets1}
 Let $M=(E,\mathcal{I})$ be a matroid and for any $k \in \mathbb{N}$, $0<k<r(E)$, $M_k=(E,\mathcal{I}_k)$ the $k$-truncation of $M$ with rank function $r_k$. Then, $E$ is inseparable with respect to $r_k$.
\end{Lemma}

\begin{proof}
 Let $E = F_1 \dot{\cup} F_2$ with $F_1 \neq \emptyset \neq F_2$ be any partition of $E$. We have to show that $r_k(F_1) + r_k(F_2) > r_k(E)$. By definition, $r_k(E)=k$. First, let $r(F_i) \leq k$ for $i=1,2$. Then, $r_k(F_i) = r(F_i)$ and consequently, $r_k(F_1) + r_k(F_2) = r(F_1)+r(F_2) \geq r(E)>k$ due to the submodularity of $r$. Next, let w.l.o.g. $r(F_1)>k$. Then, $r_k(F_1) = k$ and, since $F_2 \neq \emptyset$, $r_k(F_2)>0$. Thus,  $r_k(F_1) + r_k(F_2)= k + r_k(F_2) >k$.
\end{proof}

\begin{Lemma} \label{L:facets2}
 Let $M=(E,\mathcal{I})$ be a matroid, $M_k=(E,\mathcal{I}_k)$ its $k$-truncation with rank function $r_k$, $\emptyset \neq F \subseteq E$, and $\bar{F} = E \setminus F$ be closed with $r(\bar{F})<k<r(E)$. Then, $F$ is $k$-inseparable with respect to $r_k$. 
\end{Lemma}

\begin{proof}
$r(\bar{F})<k$ implies $r_k(\bar{F})=r(\bar{F})$, and since beyond it $\bar{F}$ is closed with respect to $r$, it is also closed with respect to $r_k$. 
Let $F = F_1 \dot{\cup} F_2$ be a proper partition of $F$. We have to show that $r^k_k(F_1)+r^k_k(F_2) < r^k_k(F)$. First, suppose that $I \in \mathcal{I}$ with $|I|=k$ and $|I \cap \bar{F}|= r_k(\bar{F})$ implies $I \cap F_1 = \emptyset$ or $I \cap F_2 = \emptyset$.  Since $\bar{F}$ is closed with respect to $r_k$, it follows that $r_k^k(F_1)=r^k_k(F_2)=0$, while $r^k_k(F)= k - r_k(\bar{F}) >0$. So assume that there is some independent set $I'$ of cardinality $k$ such that $|I' \cap \bar{F}|=r_k(\bar{F})$ and $I' \cap F_i \neq \emptyset$ for $i=1,2$. Since $k<r(E)$, there is some element $e$ such that $I := I' \cup \{e\}$ is independent with respect to $r$. Set $I_1 := I \setminus \{f_1\}$ and $I_2 := I \setminus \{f_2\}$ for $f_1 \in I \cap F_1, f_2 \in I \cap F_2$. Then,
$r^k_k(F_1) \leq |I_1 \cap F_1|$ and $r^k_k(F_2) \leq |I_2 \cap F_2|$. Hence, $r^k_k(F_1)+r^k_k(F_2) \leq |I_1 \cap F_1|+|I_2 \cap F_2| < |I_1 \cap F_1|+|I_1 \cap F_2| = |I_1 \cap F| = r^k_k(F)$.
\end{proof}

\begin{Lemma} \label{L:facets3}
Let $M=(E,\mathcal{I})$ be a matroid, $\emptyset \neq F \subseteq E$, and $A$ the matrix whose rows are the incidence vectors of $I \in \mathcal{I}$ with $|I|=k$ that satisfy the inequality $x(F) \geq r^k(F)$ at equality. Moreover, denote by $A_F$ the submatrix of $A$ restricted to $F$. Then, rank$(A_F) = |F|$ if and only if $r^k(F) \geq 1$, $\bar{F} := E \setminus F$ is closed, and (i) $F$ is $k$-inseparable or (ii) $k<r(E)$.
\end{Lemma}

\begin{proof}
\emph{Necessity.} The inequality $x(F) \geq r^k(F)$ is valid for $P_{\mathcal{I}}^{(k)}(E)$. As is easily seen, if $r^k(F) \leq 0$, then  $\mbox{rank}(A_F) < |F|$.  Next, assume that $\bar{F}$ is not closed. Then, there is some $e \in F$ such that $r(\bar{F} \cup \{e\})=r(\bar{F})$ which is equivalent to $r^k(F)=r^k(F \setminus \{e\})$. Thus, $x(F) \geq r^k(F)$ is the sum of the inequalities $x(F \setminus \{e\}) \geq r^k(F \setminus \{e\})$ and $x_e \geq 0$. This implies $\chi^I_e=0$ for all incidence vectors of independent sets $I$ with $|I|=k$ satisfying $x(F) \geq r^k(F)$ at equality. Again, it follows $\mbox{rank}(A_F) < |F|$. Finally, suppose that
neither $k<r(E)$ nor $F$ is $k$-inseparable. Then, $k=r(E)$ and $F$ is $r(E)$-separable. Thus, the inequality $x(F) \geq r^{r(E)}(F)$ is the sum of the valid inequalities $x(F_1) \geq r^{r(E)}(F_1)$ and $x(F_2) \geq r^{r(E)}(F_2)$ for some $F_1 \neq \emptyset \neq F_2$ with $F = F_1 \dot{\cup} F_2$. Setting $\lambda := r^{r(E)}(F_2) \chi_F^{F_1}-r^{r(E)}(F_1) \chi_F^{F_2}$, we see that for any $|F| \times |F|$ submatrix $\tilde{A}_F$ of $A_F$ we have $\tilde{A}_F \lambda = 0$, that is, the columns of $\tilde{A}_F$ are linearly dependent which implies $\mbox{rank}(A_F) < |F|$.

\emph{Suffiency.} First, let $k=r(E)$. Suppose $\mbox{rank}(A_F)<|F|$.  Then, $A_F \lambda = 0$ for some $\lambda \in \mathbb{R}^F, \, \lambda \neq 0$. Since $\bar{F}$ is closed and $r^k(F) \geq 1$ (that is, $r(\bar{F})<k$), for each $e \in F$ there is an independent set $I$ with $|I|=k$ that contains $e$ and whose incidence vector satisfies $x(F) \geq r^k(F)$ at equality. Thus, $A_F$ does not contain a zero-column. Moreover, $A_F \geq 0$, and hence, $F_1 := \{e \in F : \lambda_e > 0\}$ and $F_2 := \{e
 \in F : \lambda_e \leq 0\}$ defines a proper partition of $F$. Let $J \subseteq \bar{F}$ with $|J|=r(\bar{F})$ be an independent set. For $i=1,2$, let $B_i \subseteq F$ be an independent set such that $J \cup B_i$ is a basis of $E$ and $J \cup (B_i \cap F_i)$ is a basis of $\bar{F} \cup F_i$. Set $S_i :=B_i \cap F_i$ and $T_i := B_i \setminus S_i$
($i=1,2$). By construction, $T_1 \subseteq F_2$ and $T_2 \subseteq
 F_1$. By matroid axiom (iii), to $J \cup S_1$ there is some
 $U_1 \subseteq J \cup B_2$ such that $K := J \cup S_1 \cup U_1$ is a
 basis of $F$. Clearly, $U_1 \subseteq (B_2 \cap F_2) = S_2$.
Since the incidence vectors of $J \cup B_1$ and $K$ are rows of $A$, it follows immediately $\lambda(T_1) = \lambda(U_1)$.
With an analogous 
 construction one can show that there is some $U_2 \subseteq S_1$ such
 that $\lambda(U_2)=\lambda(T_2)$. It follows,
 $\lambda(T_2) = -\lambda(S_2)  \geq -
 \lambda(U_1) = -\lambda(T_1) = \lambda(S_1) \geq
 \lambda(U_2) = \lambda(T_2)$. Thus, between all terms we have equality
 implying $\lambda(S_1) = \lambda(U_2)$. Moreover, since $U_2
 \subseteq S_1$ and $\lambda_e >0$ for all $e \in S_1$, it follows $S_1
 = U_2$. Hence, $K = J \cup S_1 \cup S_2$. This, in turn, implies that $F$ is 
$k$-separable, a contradiction.

It remains to show that the statement is true if $k<r(E)$. Let $M_k = (E,\mathcal{I}_k)$ be the $k$-truncation of $M$ with rank function $r_k$. By hypothesis, all conditions of Lemma \ref{L:facets2} hold. Hence, $F$ is $k$-inseparable with respect to $r_k$. Thus, all conditions of the lemma hold for $r_k$ instead of $r$ and hence, $\mbox{rank}(A_F) = |F|$.
\end{proof}

\begin{Theorem} \label{T:facets4}
 Let $M=(E,\mathcal{I})$ be a matroid and $k \in \mathbb{N}$, $0 < k \leq r(E)$.
\begin{enumerate}
 \item[(a)] $P_{\mathcal{I}}^{(k)}(E)$ has dimension $|E|-1$ if and only if $E$ is inseparable or $k < r(E)$.
 \item[(b)] Let $\dim P_{\mathcal{I}}^{(k)}(E) = |E|-1$ and $\emptyset \neq F \subsetneq E$. The inequality $x(F) \leq \rank(F)$ defines a facet of $P_{\mathcal{I}}^{(k)}(E)$ if and only if $F$ is closed and inseparable, $r(F) < k$, and (i) $\bar{F} := E \setminus F$ is $k$-inseparable or (ii) $k < r(E)$.
\end{enumerate}
\end{Theorem}

\begin{proof}
(a) First, let $k=r(E)$. For any $\emptyset \neq F \subseteq E$, the rank inequality $x(F) \leq r(F)$ defines a facet of $P_{\mathcal{I}}(E)$ if and only if $F$ is closed and inseparable. Consequently, the polytope $P_{\mathcal{I}}^{(r(E))}(E)$, which is a face of $P_{\mathcal{I}}(E)$, has dimension $|E|-1$ if and only if $E$ is inseparable. Next, let $0 < k < r(E)$. By Lemma \ref{L:facets1}, $E$ is inseparable with respect to the rank function $r_k$ of the $k$-truncation $M_k=(E,\mathcal{I}_k)$. Consequently, $x(E) \leq r_k(E)=k$ defines a facet of $P_{\mathcal{I}_k}(E)$ and hence, $\dim P_{\mathcal{I}}^{(k)}(E) = |E|-1$.

\vspace{\baselineskip}
(b) Clearly, $x(F) \leq r(F)$ does not induce a facet of $P_{\mathcal{I}}^{(k)}(E)$ if $F$ is separable or not closed, since $\dim P_{\mathcal{I}}^{(k)}(E) = |E|-1$, and hence, any inequality that is not facet defining for $P_{\mathcal{I}}(E)$ is also not facet defining for $P_{\mathcal{I}}^{(k)}(E)$. Next, if $r(F) \geq k$, then holds obviously $x(F) \leq x(E)=k \leq r(F)$, that is, either $F$ is not closed, $x(F) \leq r(F)$ is an implicit equation, or the face induced by $x(F) \leq r(F)$ is the emptyset. Finally, assume that $F$ is closed but neither (i) nor (ii) holds. Then, $k=r(E)$ and $\bar{F}$ is $k$-separable. Thus, there are nonempty subsets $\bar{F}_1, \bar{F}_2$ of $\bar{F}$ with $\bar{F} = \bar{F}_1 \dot{\cup} \bar{F}_2$ such that $r^k(\bar{F})= r^k(\bar{F}_1)+ r^k(\bar{F}_2)$. Now, the inequality $x(\bar{F}) \geq r^{k}(\bar{F})$, which is equivalent to $x(F) \leq r(F)$, is the sum of the valid inequalities $x(\bar{F}_i) \geq r^k(\bar{F}_i)$, $i=1,2$, both not being implicit equations.

To show the converse, let $F$ satisfy all conditions mentioned in Theorem \ref{T:facets4} (b). The restriction of $M=(E,\mathcal{I})$ to $F$ is again a matroid. Denote it by $M'=(F, \mathcal{I}')$ and its rank function by $r'$. $F$ remains inseparable with respect to $r'$. Thus, the restriction of $x(F) \leq r(F)$ to $F$, denoted by $x_F(F) \leq r(F)=r'(F)$, induces a facet of $P_{\mathcal{I}'}(F)$. A set of affinely independent vectors whose sum of components is equal to some $\ell$, is also linearly independent. Thus, there are $|F|$ linearly independent vectors $\chi^{I'_j}$ of independent sets $I'_j \in \mathcal{I'}$ of cardinality $r'(F)$ ($j=1,\ldots,|F|$). The sets $I'_j$ are also independent sets with respect to $\mathcal{I}$. Due to the matroid axiom (iii), $P:=I'_1$ can be completed to an independent set $I_1$ of cardinality $k$. Since $P \subseteq F$ and $|P|=r(F)$, $Q := I_1 \setminus P \subseteq \bar{F}$. Now, $I'_j, I_1 \in \mathcal{I}$, $I'_j \subseteq F$, and $r(F) = |I'_j| < |I_1|=k$. Hence, $I_j := I'_j \cup Q \in \mathcal{I}$ for all $j$. Consequently, we have $|F|$ linearly independent vectors $\chi^{I_j} \in P_{\mathcal{I}}^{(k)}(E)$ satisfying $x(F) \leq r(F)$ at equality.

Next, let $A$ be the matrix whose rows are the incidence vectors of tight independent sets and $A_{\bar{F}}$ its restriction to $\bar{F}$. By Lemma \ref{L:facets3}, $A_{\bar{F}}$ contains a $|\bar{F}| \times |\bar{F}|$ submatrix $B$ of full rank. By construction, each row $B_i$ of $B$ is an incidence vector of an independent set $J'_i \subseteq \bar{F}$ with $|J'_i|=r^k(\bar{F})$. W.l.o.g. we may assume that $B_1 = \chi^{Q}$, that is, $Q = J'_1$. By a similar argument as above, the independent sets $J_i := J'_i \cup P$ are tight and its incidence vectors are linearly independent. 

Alltogether we have $|F|$ linearly independent vectors $\chi^{I_j}$ with $I_j \cap \bar{F} = Q$ and $|\bar{F}|$ linearly independent vectors $\chi^{J_i}$ with $J_i \cap F = P$, where $J_1=I_1$. As is easily seen, this yields a system of $|F|+|\bar{F}|-1 = |E|-1$ linearly independent vectors satisfying $x(F) \leq r(F)$ at equality.
\end{proof}

\begin{Theorem} \label{T:facets5}
$\MAT$ is fulldimensional unless $c=(0,r(E))$ and $E$ is separable.
\end{Theorem}

\begin{proof}
Clearly, $\dim \MAT \geq \dim P_{\mathcal{I}}^{(c_p)}(E) + 1$ for all $p$, since the equation $x(E) = c_p$ is satisfied by all $y \in P_{\mathcal{I}}^{(c_p)}(E)$ but violated by at least one vector $z \in \MAT$.

If $0 < c_p < r(E)$ for some $p$, then, by Theorem \ref{T:facets4}, $\dim P_{\mathcal{I}}^{(c_p)}(E) = |E|-1$, and consequently $\dim \MAT = |E|$. If there is no such $p$, then $c = (0,r(E))$. Again by Theorem \ref{T:facets4}, $\dim P_{\mathcal{I}}^{(r(E))}(E) = |E|-1$ if and only if $E$ is inseparable. Since $\dim P_{\mathcal{I}}^{(0,r(E))}(E) = \dim P_{\mathcal{I}}^{(r(E))}(E) + 1$, it follows the claim.
\end{proof}

\begin{Theorem} \label{T:facets6}
For any $\emptyset \neq F \subseteq E$, the rank inequality $x(F) \leq r(F)$ defines a facet of $\MAT$ if and only if one of the following conditions holds.
\begin{enumerate}
\item[(i)] $0<r(F)<c_{m-1}$ and $F$ is closed and inseparable.
\item[(ii)] $0<c_{m-1}=r(F)<c_m<r(E)$, and $F$ is closed and inseparable.
\item[(iii)] $0<c_{m-1}=r(F)<c_m=r(E)$, $F$ is closed and inseparable, $\bar{F}$ is $c_m$-inseparable, and $E$ is inseparable.
\item[(iv)] $0<c_{m-1}<c_m=r(F)$, $F=E$, and $c_m<r(E)$ or $E$ inseparable.
\item[(v)] $c_{m-1}=c_1=0$, $c_m=r(E)$, and $r(F)+r(E \setminus F)=r(E)$.
\end{enumerate}
\end{Theorem}

\begin{proof}
We prove the theorem by case by case enumeration.

(a) Let $0<r(F)<c_{m-1}$. It is not hard to see that if $F$ is separable or not closed, then $x(F) \leq r(F)$ does not define a facet of $\MAT$. So, let $F$ be closed and inseparable. By Theorem \ref{T:facets4}, $x(F) \leq r(F)$ defines a facet of $P_{\mathcal{I}}^{(c_{m-1})}(E)$ and $\dim P_{\mathcal{I}}^{(c_{m-1})}(E)= |E|-1$. Thus, it defines also a facet of $\MAT$.

(b) Let $0<c_{m-1}=r(F)<c_m<r(E)$. Clear by interchanging $c_{m-1}$ and $c_m$ in item (a).

(c) Let $0<c_{m-1}=r(F)<c_m=r(E)$. The conditions mentioned in (iii) are equivalent to the postulation that $x(F) \leq r(F)$ defines a facet of $P_{\mathcal{I}}^{(c_m)}(E)$ and $\dim P_{\mathcal{I}}^{(c_m)}(E)=|E|-1$. If, indeed, the latter is true, then $x(F) \leq r(F)$ induces a facet also of $\MAT$. To show the converse, suppose, for the sake of contradiction, that $x(F) \leq r(F)$ does not induce a facet of $P_{\mathcal{I}}^{(c_m)}(E)$ or $\dim P_{\mathcal{I}}^{(c_m)}(E)<|E|-1$. Let $\mathcal{B} := \{\chi^{I_j} : I_j \in \mathcal{I}, |I_j|=c_m, j=1,\ldots, z,\}$ be an affine basis of the face of $P_{\mathcal{I}}^{(c_m)}(E)$ induced by $x(F) \leq r(F)$. By hypothesis, $z \leq |E|-2$. Moreover, set $J := I_1 \cap F$ and $K:= I_1 \setminus J$. Then, any incidence vector of an independent set $L \subseteq F$ with $|L|=c_{m-1}$ can be obtained as an affine combination of the set $\mathcal{B}' := \mathcal{B} \cup \{\chi^J\}$, which can be seen as follows: $L,I_1 \in \mathcal{I}$, and $|L|=r(F)$ implies $L \cup {K} \in \mathcal{I}$. Consequently, $\chi^L = \chi^{L \cup K} - \chi^K$. Now, $\chi^K = \chi^{I_1} - \chi^J$ and $\chi^{L \cup K}= \sum_{j=1}^z \lambda_j \chi^{I_j}$ with $\sum_{j=1}^z \lambda_j =1$, since $L \cup K$ is tight. Thus, $\chi^L= \sum_{j=1}^z \lambda_j \chi^{I_j}-\chi^{I_1} + \chi^J$, that is, $\chi^L$ is in the affine hull of $\mathcal{B}'$. Since $|\mathcal{B}'| \leq |E|-1$, $x(F)  \leq r(F)$ is not facet defining for $\MAT$, a contradiction.

(d) Let $0<c_{m-1} < r(F)<c_m$. Since none of the independent sets $I$ with $|I|=c_p$ is tight for $p=1,\ldots,m-1$, $x(F) \leq r(F)$ defines a facet of $\MAT$ if and only if it is an implicit equation for $P_{\mathcal{I}}^{(c_m)}(E)$ and $\dim P_{\mathcal{I}}^{(c_m)}(E)=|E|-1$. However, $\dim P_{\mathcal{I}}^{(c_m)}(E)=|E|-1$ implies $c_m<r(E)$ or $E$ is inseparable. In either case, it follows that $x(F) \leq r(F)$ is an implicit equation for $P_{\mathcal{I}}^{(c_m)}(E)$ if and only if $F=E$. Thus, $r(F)=c_m$, a contradiction.

(e) Let $0<c_{m-1}<c_m=r(F)$. Clearly, if $F \subset E$, then $x(F) \leq r(F)$ is strictly dominated by the cardinality bound $x(E) \leq c_m$. Consequently, $F=E$ and $x(F) \leq r(F)$ is an implicit equation for $P_{\mathcal{I}}^{(c_m)}(E)$. For the same reasons as in (d), $\dim P_{\mathcal{I}}^{(c_m)}(E)=|E|-1$. Hence, $c_m<r(E)$ or $E$ is inseparable.

(f) Let $c_{m-1}=c_1=0$. Again, $x(F) \leq r(F)$ defines a facet of $\MAT$ if and only if it is an implicit equation for $P_{\mathcal{I}}^{(c_m)}(E)$. This is the case if and only if $c_m=r(E)$ and $r(F)+r(E \setminus F)=r(E)$.

(g) Let $r(F)>c_m$. Then, $x(F) \leq x(E)\leq c_m<r(F)$, that is, the face induced by $x(F) \leq r(F)$ is the empty set.
\end{proof}

\begin{Theorem} \label{T:facets7}
Let $F \subseteq E$ with $c_p< r(F) < c_{p+1}$ for some $p \in
\{1,\dots,m-1\}$. Then, the
rank induced forbidden set inequality $\FS_F(x) 
\leq c_p (c_{p+1}-r(F))$ defines a facet of $\MAT$ if
and only if
\begin{enumerate}
\item[(a)] $c_p=c_1=0$ and the inequality $x(F) \leq r(F)$ defines a facet of $P_{\mathcal{I}}^{(c_{p+1})}(E)$, or
\item[(b)] $c_p>0$, $F$ is closed and (i) $\bar{F} := E \setminus F$ is $c_{p+1}$-inseparable or (ii) $c_{p+1}<r(E)$.
\end{enumerate}

\end{Theorem}

\begin{proof}
For $P_{\mathcal{I}}^{(c_{p+1})}(E)$, the inequality $\FS_F(x) \leq c_p (c_{p+1}-r(F))$ is equivalent to $x(F) \leq r(F)$, while for $P_{\mathcal{I}}^{(c_{p})}(E)$, it is equivalent to $x(F) \leq c_p$. Hence,
in case $c_p=c_1=0$, $\FS_F(x) \leq c_p (c_{p+1}-r(F))$ induces a facet of $\MAT$ if and only if it induces a facet of $P_{\mathcal{I}}^{(c_{p+1})}(E)$. When $\dim P_{\mathcal{I}}^{(c_{p+1})}(E) = |E|-1$, this is the case if and only if $F$ is closed and inseparable and (i) $\bar{F}$ is $c_{p+1}$-inseparable or (ii) $c_{p+1} < r(E)$, see Theorem \ref{T:facets4} (b).

In the following, let $c_p>0$. Let $A$ be the matrix whose rows are the incidence vectors of $I \in \mathcal{I}$ with $|I|=c_p$ or $|I|=c_{p+1}$ that satisfy the inequality $\FS_F(x) \leq c_p (c_{p+1}-r(F))$ at equality. Denote by $A_F$ and $A_{\bar{F}}$ the restriction of $A$ to $F$ and $\bar{F}$, respectively. By Theorem \ref{T:facets5}, $\MAT$ is fulldimensional. Hence, $\FS_F(x) \leq c_p (c_{p+1}-r(F))$ is facet defining if and only if the affine rank of $A$ is equal to $|E|$. 

If $F$ is not closed, then there is some $e \in \bar{F}$ with $r(F \cup \{e\}) = r(F)$. Thus, $\FS_{F'}(x) \leq c_p (c_{p+1}-r(F'))$ is a valid inequality for $\MAT$, where $F' := F \cup \{e\}$, and $\FS_F(x) \leq c_p (c_{p+1}-r(F))$ is the sum of this inequality and $- (c_{p+1} - c_p) x_e \leq 0$. Next, assume that neither (i) nor (ii) holds. Then, $c_{p+1}=r(E)$ and $\bar{F}$ is $r(E)$-separable. Thus, there is a proper partition $\bar{F} = \bar{F}_1 \dot{\cup} \bar{F}_2$ of $\bar{F}$ with $r^{r(E)}(\bar{F}_1) + r^{r(E)}(\bar{F}_2) = r^{r(E)}(\bar{F})$. Since $F$ is closed, it is not hard to see that $r^{r(E)}(\bar{F}_i) >0$ which implies $c_p < r(F \cup \bar{F}_i) < r(E)$ for $i=1,2$, and hence, the inequalities $\FS_{F \cup \bar{F}_1}(x) \leq c_p (c_{p+1}-r(F \cup \bar{F}_1))$ and $\FS_{F \cup \bar{F}_2}(x) \leq c_p (c_{p+1}-r(F \cup \bar{F}_2))$ are valid. One can check again that then $\FS_F(x) \leq c_p (c_{p+1}-r(F))$ is the sum of these both rank induced forbidden set inequalities.

To show the converse, let
$M^F = (F, \mathcal{I}^F)$ with $\mathcal{I}^F := \{I \cap F : I \in \mathcal{I}\}$ be the restriction of $M$ to $F$ and $M^F_{c_p} = (F,\mathcal{I}^F_{c_p})$ the $c_p$-truncation of $M^F$. Since $0 < c_p < r(F)$, Lemma \ref{L:facets1} implies that $F$ is inseparable with respect to the rank function of $M^F_{c_p}$. Consequently, the restriction of $x(F) \leq c_p$ to $F$ defines a facet of $P_{\mathcal{I}^F_{c_p}}(F)$. Hence, $A$ contains a $|F| \times |E|$ submatrix $B$ such that $B_F$ is nonsingular and $B_{\bar{F}} = 0$. Next, since $F$ is closed, $r^{c_{p+1}}(\bar{F}) \geq 1$, and (i) $\bar{F}$ is $c_{p+1}$-inseparable or (ii) $c_{p+1}<r(E)$, Lemma \ref{L:facets3} implies that $A$ contains a $|\bar{F}| \times |E|$ submatrix $C$ such that $C_{\bar{F}}$ is nonsingular. Thus,
\[D :=
\begin{pmatrix}
B_F & 0 \\
C_F & C_{\bar{F}}
\end{pmatrix} 
\]
is a nonsingular $|E| \times |E|$ submatrix of $A$ (or a row permutation of $A$).
\end{proof}

\subsection{Separation problem}
Given any $\MAT$ and any $x^* \in \mathbb{R}^E$, the separation
problem consists of finding an inequality among \eqref{eq_FS}-\eqref{eq_nn}
violated by $x^*$ if there is any. This problem should be solvable
efficiently, due to the polynomial time equivalence of optimization and separation
(see Gr\"otschel, Lov\'{a}sz, and Schrijver~\cite{GLS}). By default,
we may assume that $x^*$ satisfies the cardinality bounds
\eqref{eq_lowerBound}, \eqref{eq_upperBound} and the nonnegativity
constraints \eqref{eq_nn}. A violated rank inequality among
\eqref{eq_rankInequalities} (if there is any) can
be found by a polynomial time algorithm proposed by
Cunningham~\cite{Cunningham}. So, we are actually interested only in
finding an efficient algorithm that solves the separation problem for
the class of rank induced forbidden set inequalities \eqref{eq_FS}.
If $\rank(F)=|F|$ for all $F \subseteq E$, then the separation routine proposed by Gr\"otschel~\cite{Groetschel}
can be applied: For each forbidden cardinality $k$ one just needs to take the first $k$ greatest weights, say  $x^*_{e_1},\ldots, x^*_{e_k},$ and check whether the forbidden set inequality associated with $F:= \{e_1,\ldots,e_k\}$ is violated by $x^*$.
Otherwise we shall see
that the separation problem for the rank induced forbidden set inequalities
can be transformed to that for the rank inequalities. 

The separation problem for the class of rank induced forbidden set
inequalities consists of checking whether or not
\begin{equation*}
\begin{array}{rcll}
\multicolumn{4}{l}{
(c_{p+1}-\rank(F)) x^*(F)-
(\rank(F)-c_p)x^*(E \setminus F) \; \leq \;
c_p(c_{p+1}-\rank(F))}\\ 
\multicolumn{4}{r}{\mbox{for all $F \subseteq E$ with $c_p <
    \rank(F) < 
  c_{p+1}$} \mbox{ for some $p \in \{0,\dots,m-1\}$}.}
\end{array}
\end{equation*}
For any $F \subseteq E$,
\begin{equation*}
\begin{array}{cl}
&(c_{p+1}-\rank(F)) x^*(F)-(\rank(F)-c_p)x^*(E \setminus F) \leq 
c_p(c_{p+1}-\rank(F))\\[0.2cm]
\Leftrightarrow & (c_{p+1}-c_p) x^*(F)-(\rank(F)-c_p)x^*(E) \leq 
c_p(c_{p+1}-\rank(F))\\[0.2cm]
\Leftrightarrow & x^*(F) \leq \frac{
c_p(c_{p+1}-\rank(F))+(\rank(F)-c_p)x^*(E)}{ (c_{p+1}-c_p)} =: \gamma_F. 
\end{array}
\end{equation*}
Moreover, for any
$k \in \{1,\dots,\rank(E)\}$, the right hand sides of the
inequalities $x^*(F) \leq \gamma_F$ for $F \subseteq E$ with
$\rank(F)=k$ are equal and differ only by a constant to the right hand
sides of the corresponding rank inequalities $x(F) \leq \rank(F)=k$.
Thus, both the 
separation problem for the rank inequalities and rank induced forbidden set
inequalities could be solved by finding, for each $k \in
\{1,\dots,|E|\}$, a set $F^* \subseteq E$ of rank $k$ that maximizes
$x^*(F)$. If $x^*(F^*)>k$, then the inequality $x(F^*)\leq \rank(F^*)$
is violated by $x^*$. If, in addition, $c_p<k<c_{p+1}$ for some $p \in
\{1,\dots,m-1\}$ and $x^*(F^*)> \gamma_{F^*}$, then $x^*$ violates the
rank induced forbidden set inequality associated with $F^*$.

This natural generalization of Gr\"otschel's separation algorithm, however, seems usually not to result in an efficient separation routine. In order to mark the difficulties, we investigate
the above approach for the class of rank inequalities, when
$M=(E,\mathcal{I})$ is the graphic matroid defined on some graph
$G=(V,E)$. It is well known that the closed and inseparable rank
inequalities for the graphic matroid are of the form $x(E(W)) \leq
|W|-1$ for $\emptyset \neq W \subseteq V$. 
If we would tackle the
separation problem for this class of inequalities by finding, for each
$k \in \{1,\dots, |W|\}$ separately, a set $W^*_k$ that maximizes
$x^*(E(W))$ such that $|W|=k$, then we would run into trouble, since
for each $k$, such 
a problem is the weighted version of the densest $k$-subgraph problem
which is known to be NP-hard (see Feige and
Seltser~\cite{FS}). 

The last line of argument indicates that it is probably not a good
idea to split the separation problem for the rank induced forbidden set inequalities
\eqref{eq_FS} into separation problems for the
subclasses $\FS_F(x) \leq c_p(c_{p+1}-\rank(F))$ with $\rank(F)=k, \, k
\in \{c_1+1,\dots, c_m-1\} \setminus \{c_2,c_3,\dots,c_{m-1}\}$. It
would be rather better to approach it as ``non-cardinality
constrained'' problem. And this is exactly what Cunningham did
for the rank inequalities. 

In the sequel, we firstly remind of some important facts regarding
Cunningham's algorithm for the separation of the rank
inequalities. Afterwards, we show how the separation 
problem for the rank induced forbidden set inequalities can be reduced to
that for the rank inequalities.

The theoretical background of Cunningham's separation routine is the
following min-max relation.
\begin{Theorem}[Edmonds~\cite{Edmonds1970}] \label{Tmm}
For any $x^* \in \mathbb{R}^E_+$, $\max \{y(E) : y \in P_M(E), y \leq x^*\}=
\min \{\rank(F)+x^*(E \setminus F) : F \subseteq E\}$. \hfill $\Box$
\end{Theorem}
Indeed, for any $y \in P_M(E)$ with $y \leq x^*$, $y(E)=y(F)+y(E \setminus
F) \leq \rank(F)+x^*(E \setminus F)$, and equality will be
attained if only if $y(F)=\rank(F)$ and $y(E \setminus F)= x^*(E \setminus
F)$. Theorem \ref{Tmm} guarantees that any $F$ minimizing
$\rank(F)+x^*(E \setminus F)$ maximizes $x^*(F)-\rank(F)$. For any 
matroid $M=(E,\mathcal{I})$ given by an independence testing oracle
and any $x^* \in \mathbb{R}^E_+$, Cunningham's algorithm finds a $y
\in P_M(E)$ with $y \leq x^*$ maximizing $y(E)$, a decomposition of
$y$ as convex combination of incidence vectors of independent sets,
and a set $F^* \subseteq E$ with $\rank(F^*)+x^*(E \setminus
F^*)=y(E)$ in strongly polynomial time. The vector $y$ will be
constructed by path
augmentations along shortest paths in an auxiliary digraph.

Next, we return to the separation problem for the rank induced forbidden set
inequalities \eqref{eq_FS}. In the sequel, we suppose that $x^*$
satisfies the rank inequalities \eqref{eq_rankInequalities}. 
\begin{Lemma} \label{L_SEP1}
Let $x^* \in \mathbb{R}^E_+$ satisfying all rank inequalities
\eqref{eq_rankInequalities}.   
If a rank induced forbidden set inequality $\FS_F(x) \leq
c_p(c_{p+1}-\rank(F))$ with $c_p<\rank(F)<c_{p+1}$ is violated by
$x^*$, then $c_p < x^*(E) < c_{p+1}$. 
\end{Lemma}

\begin{proof}
First, assume that $x^*(E) \leq c_p$. Then $x^*(F) \leq c_p$, and hence,
\[
\begin{array}{cl}
& (c_{p+1}-\rank(F)) x^*(F) - (\rank(F)-c_p)x^*(E \setminus F)\\
 \leq &  (c_{p+1}-\rank(F)) c_p - (\rank(F)-c_p)x^*(E \setminus F)\\
 \leq &  c_p(c_{p+1}-\rank(F)).
\end{array}
\]

Next, assume that $x^*(E) \geq c_{p+1}$.
By hypothesis, $x^*$ satisfies all rank inequalities
\eqref{eq_rankInequalities}, in particular, $x(F) \leq \rank(F)$. Thus,
\[
\begin{array}{cl}
&(c_{p+1}-\rank(F)) x^*(F) - (\rank(F)-c_p)x^*(E \setminus F)\\
 = &  (c_{p+1}-c_p) x^*(F) - (\rank(F)-c_p)x^*(E)\\
 \leq &  (c_{p+1}-c_p) \rank(F) - (\rank(F)-c_p)x^*(E)\\
 \leq &  (c_{p+1}-c_p) \rank(F) - (\rank(F)-c_p)c_{p+1}\\
 = & c_p(c_{p+1}-\rank(F)).
\end{array}
\]
\end{proof}

\begin{Lemma} \label{L_SEP2}
Let $x^* \in \mathbb{R}^E_+$ satisfying all rank inequalities
\eqref{eq_rankInequalities}, and 
let $c_p<x^*(E)<c_{p+1}$ for some $p \in \{1,\dots,m-1\}$. Then for
any $F \subseteq E$ we have: If $(c_{p+1}-\rank(F)) x^*(F) -
(\rank(F)-c_p)x^*(E \setminus F) > c_p(c_{p+1}-\rank(F))$, then
$c_p<\rank(F)<c_{p+1}$.   
\end{Lemma}

\begin{proof}
Let $F \subseteq E$, and assume that $\rank(F) \leq c_p$. Then,
\[
\begin{array}{cl}
&(c_{p+1}-\rank(F)) x^*(F) - (\rank(F)-c_p)x^*(E \setminus F)-
c_p(c_{p+1}-\rank(F))\\
=& (c_{p+1}-c_p) x^*(F) - (\rank(F)-c_p)x^*(E)-
c_p(c_{p+1}-\rank(F))\\
\leq & (c_{p+1}-c_p) \rank(F) - (\rank(F)-c_p)x^*(E)-
c_p(c_{p+1}-\rank(F))\\
= & \underbrace{(c_{p+1}-x^*(E))}_{>
  0}\underbrace{(\rank(F)-c_p)}_{\leq 0}\; \leq \; 0.
\end{array}
\]

Next, if  $\rank(F) \geq c_{p+1}$, then 
\[
\begin{array}{cl}
&(c_{p+1}-\rank(F)) x^*(F) - (\rank(F)-c_p)x^*(E \setminus F)-
c_p(c_{p+1}-\rank(F))\\
=& (c_{p+1}-c_p) x^*(F) - (\rank(F)-c_p)x^*(E)-
c_p(c_{p+1}-\rank(F))\\
\leq & (c_{p+1}-c_p) x^*(E) - (\rank(F)-c_p)x^*(E)-
c_p(c_{p+1}-\rank(F))\\
= & \underbrace{(c_{p+1}-\rank(F))}_{\leq
  0}\underbrace{(x^*(E)-c_p)}_{> 0}\; \leq \; 0.
\end{array}
\]

Thus, $(c_{p+1}-\rank(F)) x^*(F) -
(\rank(F)-c_p)x^*(E \setminus F) > c_p(c_{p+1}-\rank(F))$ at most if
$c_p<\rank(F)<c_{p+1}$.   
\end{proof}

\begin{Theorem} \label{T_SEP}
Given a matroid $M=(E,\mathcal{I})$ by an independence testing oracle, a cardinality sequence $c$, and a vector $x^\star \in \mathbb{R}^E_+$ satisfying all rank
inequalities \eqref{eq_rankInequalities}, 
the separation problem for $x^*$ and the rank induced forbidden set
inequalities \eqref{eq_FS} can be solved in strongly polynomial time.
\end{Theorem}

\begin{proof}
By Lemmas \ref{L_SEP1} and \ref{L_SEP2} we know that $x^*$ violates a
rank induced forbidden set inequality at most if  $c_p<x^*(E)<c_{p+1}$ for
some $p \in \{1,\dots,m-1\}$. Thus, if $x^*(E)=c_q$ for some $q \in
\{1,\dots,m\}$, then $x^* \in \MAT$.

Suppose that $c_p<x^*(E)<c_{p+1}$ for some $p \in \{1,\dots,m-1\}$.
We would like to find some $F' \subseteq E$ such that 
\[(c_{p+1}-\rank(F')) x^*(F') - (\rank(F')-c_p)x^*(E \setminus F') -
c_p(c_{p+1}-\rank(F')) >0\] if there is any. Lemma \ref{L_SEP2}
says that  $c_p<\rank(F')<c_{p+1}$, and thus, the inequality $\FS_{F'}(x)
\leq c_p(c_{p+1}-\rank(F'))$ is indeed a rank induced forbidden set
inequality among \eqref{eq_FS} violated by $x^*$. If there is
no such $F'$, then for all $F \subseteq E$ with $c_p<\rank(F)<c_{p+1}$
the associated rank induced forbidden set inequality with $F$ is satisfied by
$x^*$, and by Lemma \ref{L_SEP1}, all other rank induced forbidden set
inequalities among \eqref{eq_FS} are also satisfied by $x^*$.

To find such a subset $F'$ of $E$,
set $\delta := \frac{x^*(E)-c_p}{c_{p+1}-c_p}$. Since
$c_p<x^*(E)<c_{p+1}$, $0<\delta<1$. Moreover, $
\frac{c_{p+1}-x^*(E)}{c_{p+1}-c_p}=1-\delta$. For any  $F
\subseteq E$ it now follows:
\[
\begin{array}{crcl}
 & (c_{p+1}-c_p) x^*(F) - (\rank(F)-c_p)x^*(E)-
c_p(c_{p+1}-\rank(F)) & > & 0\\
\Leftrightarrow & x^*(F) - \frac{\rank(F)x^*(E)+c_px^*(E)-
c_p c_{p+1}+c_p\rank(F)}{ c_{p+1}-c_p} & > & 0\\
\Leftrightarrow & x^*(F) - \rank(F)\frac{x^*(E)-c_p}{c_{p+1}-c_p} -
c_p\frac{c_{p+1}-x^*(E)}{ c_{p+1}-c_p} & > & 0\\
\Leftrightarrow & x^*(F) - \rank(F)\delta  & > & c_p(1-\delta)\\
\Leftrightarrow & \frac{x^*(F)}{\delta} - \rank(F)  & > &
c_p\frac{(1-\delta)}{\delta}.\\ 
\end{array}
\]
Setting $x' := \frac{1}{\delta} x^*$, we see that the last inequality
is equivalent to $x'(F) - \rank(F) > c_p\frac{(1-\delta)}{\delta}$. Thus,
we can apply Cunningham's algorithm to find some $F \subseteq E$ that
maximizes $x'(F)-\rank(F)$. If  $x'(F)-\rank(F)>
c_p\frac{(1-\delta)}{\delta}$, then $c_p<\rank(F)<c_{p+1}$ and the
rank induced forbidden set inequality associated with $F$ is violated by $x^*$.
\end{proof}

Consequently, we suggest a separation routine that works as
follows. Assume that the fractional point $x^*$ satisfies the nonnegativity constraints and the cardinality bounds. 
First, compute with Cunningham's algorithm a subset $F$ of
$E$ maximizing $x^*(F)-\rank(F)$. If $x^*(F)-\rank(F)>0$, then the
associated rank inequality  $x(F) \leq \rank(F)$ is violated by
$x^*$. 
If $x^*(F)-\rank(F) \leq 0$, then $x^*$
satisfies all rank inequalities \eqref{eq_rankInequalities}, and if,
in addition, $x^*(E)=c_p$ for some $p$, then we know that $x^* \in
\MAT$. Otherwise, i.e., if $c_p<x^*(E)<c_{p+1}$ for some $p \in
\{1,\dots,m-1\}$, then we check whether or not there is a violated
rank induced forbidden set inequality among \eqref{eq_FS} by applying
Cunningham's algorithm on $M=(E,\mathcal{I})$ and $x' =
\frac{1}{\delta}x^*$.

\begin{Corollary} \label{C:SEP2}
Given a matroid $M=(E,\mathcal{I})$ by an independence testing oracle, a cardinality sequence $c$, and a vector $x^\star \in \mathbb{R}^E_+$,  the separation problem for $x^\star$ and $\MAT$ can be solved in strongly polynomial time.
\hfill $\Box$
\end{Corollary}

\section{Concluding remarks}
\label{Sec:remarks}

The cardinality constrained matroid polytope turns out to be a useful object to enhance the theory of polyhedra
associated with cardinality constrained combinatorial optimization problems. Imposing cardinality constraints on a combinatorial optimization problem does not necessarily turn it into a harder problem: The cardinality constrained version of the maximum weight independent set problem in a matroid is manageable on the algorithmic as well as on the polyhedral side without any difficulties. Facets related to cardinality restrictions (rank induced forbidden set inequalities) are linked to well known notions of matroid theory (closed subsets of $E$). The analysis of the separation problem for the rank induced forbidden set inequalities discloses that it is sometimes better not to split a cardinality constrained problem into ``simpler'' cardinality constrained problems but to transform it into one or more non-cardinality restricted problems.

It stands to reason to investigate the intersection of two matroids
with regard to cardinality restrictions. As it is well known, if an
independence system $\mathcal{I}$ defined on some ground set $E$ can be
described as the intersection of two matroids $M_1=(E,\mathcal{I}_1)$
and $M_2=(E,\mathcal{I}_2)$, then the optimization problem $\max w(I),
\, I \in \mathcal{I}$ can be solved in polynomial time, for instance
with Lawler's weighted matroid intersection algorithm~\cite{Lawler}.
This algorithm
solves also the cardinality constrained version $\max w(I), \, I \in
\mathcal{I} \cap \CHS^c(E)$, since for each cardinality $p \leq
\rank(E)$ it generates an independent set $I$ of cardinality $p$
which is optimal among all independent sets $J$ of cardinality
$p$. Thus, from an algorithmic point of view the problem is well
studied. However, there is an open question regarding the associated
polytope. As it is well known, $P_{\mathcal{I}}(E) 
=P_{\mathcal{I}_1}(E) \cap P_{\mathcal{I}_2}(E)$, that is, the
non-cardinality constrained independent set polytope 
$P_{\mathcal{I}}(E)$ is determined by the nonnegativity constraints
$x_e \geq 0$, $e \in E$, and the rank inequalities $x(F) \leq r_j(F)$,
$\emptyset \neq F \subseteq E$, $j=1,2$, where $r_j$ is the rank function with
respect to $\mathcal{I}_j$. We do not know, however, whether
or not $P_{\mathcal{I}}^c(E)=P_{\mathcal{I}_1}^c(E) \cap
P_{\mathcal{I}_2}^c(E)$ holds. So far, we have not found any
counterexample contradicting the hypothesis that equality holds.

\end{document}